\documentclass{article}
\usepackage{amsmath,amsthm,amssymb,amscd}

\newcommand{\ga}{\alpha}
\newcommand{\gb}{\beta}
\renewcommand{\gg}{\gamma}
\newcommand{\gd}{\delta}
\newcommand{\gw}{\omega}

\newcommand{\gS}{\Sigma}
\newcommand{\gs}{\sigma}

\newcommand{\dom}{\mathrm{dom}}

\newcommand{\namba}{N}

\newtheorem{theorem}{Theorem}[section]

\newtheorem{claim}[theorem]{Claim}

\theoremstyle{definition}

\title{Bounded Namba forcing axiom may fail\footnote{2000 AMS subject classification 03E57.}}

\author{
Jind{\v r}ich Zapletal
\thanks{The authors was partially supported by NSF grant  DMS 1161078.}
\\University of Florida
}

\begin{document}
\maketitle

\begin{abstract}
I show that in a $\gs$-closed forcing extension, the bounded forcing axiom for Namba forcing fails. This answers a question of Justin Tatch Moore.
\end{abstract}

\section{Introduction}

 Ronald Jensen proved \cite{jensen:subcomplete} that Namba forcing and certain variations of it can be iterated without adding reals. It follows that the forcing axiom for the Namba forcing is consistent with the Continuum Hypothesis. During Jensen's talk at Oberwolfach set theory meeting in February 2017 on this result, Justin Tatch Moore asked whether it is possible that the Continuum Hypothesis (possibly under a large cardinal assumption) outright implies the forcing axiom or the bounded forcing axiom for Namba forcing. The question proved to be a surprisingly tough nut to crack. The purpose of this note is to provide a complete negative answer to the question:

\begin{theorem}
Assume the Continuum Hypothesis. Then in a $\gs$-closed $\aleph_2$-preserving extension the bounded forcing axiom for Namba Forcing fails.
\end{theorem}

The arguments are identical for all common versions of Namba forcing. For definiteness, I will use the ``Laver-style'' version. The \emph{Namba forcing} $\namba$ is the set of all trees $T\subset \gw_2^{<\gw}$ which contain an element $t$ (the \emph{trunk}) such that every node on $T$ is compatible with $t$ and every element of $T$ extending $t$ has $\aleph_2$ many immediate successors in $T$. The ordering is that of inclusion. Thus, the Namba forcing adds an $\gw$-sequence of ordinals cofinal in $\gw_2^V$; below, the Namba name for it will be denoted by $\dot x$.

Recall that if $P$ is a poset then the \emph{bounded forcing axiom for $P$} \cite{bagaria:bounded, shelah:bounded} is the statement that the structure $\langle H_{\aleph_2}, \in, \gw_1\rangle$ of the ground model is forced to be a $\gS_1$-elementary submodel of $\langle H_{\aleph_2}, \in, \gw_1\rangle$ of the $P$-extension. There are many equivalent restatements of this condition; under the Continuum Hypothesis, it is equivalent to the statement that $P$ does not add any branches to trees of height and width $\gw_1$ which have no branches in the ground model.

The notation used in this note follows the set theory standard of \cite{jech:newset}.

\section{The proof}

I will start with a piece of terminology. Let $U\subset 2^{<\gw_1}$ be a tree. A set $a$ of branches of $U$ is \emph{consonant} at $\gb\in\gw_1$ if for every $b\in a$, $b(\gb)=0$; otherwise, the set is \emph{dissonant} at $\gb$. The set $a$ is \emph{eventually dissonant} if there is a countable ordinal $\ga$ such that $a$ is dissonant at all larger countable ordinals. The idea of the proof is to 
force with a $\gs$-closed forcing $P$ a tree $U\subset 2^{<\gw_1}$ such that every infinite set $a\subset 2^{\gw_1}$ of its uncountable branches is eventually dissonant. At the same time, in the $P$-extension the Namba forcing will add a countably infinite set of ground model branches through $U$ which is not eventually dissonant. 

Towards the definition of the poset $P$, a condition will be a tuple $p=\langle \ga_p, e_p, f_p, a_p, b_p\rangle$ such that

\begin{itemize}
\item $\ga_p$ is a countable ordinal, $e_p\subset\gw_2$ is a countable set;
\item $f_p\colon\ga_p\times e_p\to 2$ is a function;
\item $a_p$ is a countable subset of $2^{\leq\ga_p}$ such that for no ordinal $\gd\in e_p$ and no $t\in a_p$ it is the case that for all $\gb\in\dom(t)$, $t(\gb)=f(\gb, \gd)$;
\item $b_p$ is a countable collection of infinite subsets of $e_p$.
\end{itemize}

\noindent The ordering is defined by $q\leq p$ just in case $\ga_p\leq \ga_q$, $e_p\subseteq e_q$, $f_p\subseteq f_q$, $a_p\subseteq a_q$, $b_p\subseteq b_q$, and for every ordinal $\gb\in\ga_q\setminus\ga_p$ and every set $a\in b_p$ there is $\gd\in a$ such that $f(\gb, \gd)=1$.

The idea behind the definition of the poset $P$ is the following. If $G\subset P$ is a generic filter, let $f\colon\gw_1\times\gw_2\to 2$ be the union $\bigcup_{p\in G}f_p$. Also, let $U\subset 2^{<\gw_1}$ be the collection of all countable binary sequences $t$ such that for some $\gd\in\gw_2$, $\forall\gb\in\dom(t)\ t(\gb)=f(\gb, \gd)$. Thus, $U$ is a tree and the functions $b_\gd=f(\cdot, \gd)$ are its branches. The $a_p$ coordinate is inserted to make sure that the tree $U$ has no other branches. The $b_p$ coordinate is inserted to make sure that any infinite set of branches of $U$ is eventually dissonant. The genericity of the whole construction will imply that in $V[G]$, the Namba forcing $\namba$ forces the set $\{b_{\dot x(n)}\colon n\in\gw\}$ to be consonant at uncountably many ordinals, where $\dot x$ is the $\namba$-name for its generic sequence of ordinals. I now proceed to verify the features of the poset $P$ one by one.

\begin{claim}
The poset $P$ is $\gs$-closed and $\aleph_2$-c.c.
\end{claim}

\begin{proof}
For the $\gs$-closedness, for every countable descending sequence of conditions its coordinatewise union is its lower bound.
For $\aleph_2$-c.c., note that if conditions $p, q\in P$ satisfy $\ga_p=\ga_q$, $a_p=a_q$, and $f_p\restriction (\ga_p\times e_p\cap e_q)=f_q\restriction (\ga_p\times e_p\cap e_q)$ then $p, q$ are compatible--their coordinatewise union will be their lower bound. The proof of $\aleph_2$-c.c.\ is completed by a standard $\Delta$-system argument using the CH assumption.
\end{proof}

\begin{claim}
$P$ forces the domain of $\bigcup_{p\in G}f_p$ to be equal to $\gw_1\times\gw_2$ and the set $\bigcup_{p\in G}b_p$ to be equal to $[\gw_2]^{\aleph_0}$.
\end{claim}

\begin{proof}
An elementary density argument.
\end{proof}

\begin{claim}
The poset $P$ forces that every branch through the tree $U$ is on the list $\{b_\gd\colon\gd\in\gw_2\}$.
\end{claim}

\begin{proof}
Suppose towards a contradiction that this fails and let $p\in P$ force $\dot c$ to be an unlisted a branch through $\dot U$.
Let $M$ be a countable elementary submodel of a large structure and $g\subset P\cap M$ be a filter generic over $M$, containing the condition $p$. It is immediate that the coordinatewise union $q$ of all conditions in the filter is again a condition in the poset $P$,
stronger than all conditions in the filter $p$, in particular $q\leq p$. Let $t=\dot c/g$. The genericity of the filter $g$ plus the assumption on the name $\dot c$ imply that $\ga_q=M\cap \gw_1$, $t\in 2^{\ga_q}$ and for no $\gd\in M\cap\gw_2$ it is the case that for all $\gb\in\ga_q$, $t(\gb)=f_q(\gb, \gd)$. Thus, the condition $r$ obtained from $q$ by simply adding $t$ into $a_r$, is indeed an element of the poset $P$, and $r\leq q\leq p$. At the same time, $r\Vdash\check t\subset\dot c$ and $\check t\notin\dot U$, contradicting the assumption that $\dot c$ was forced to be a branch through $\dot U$.
\end{proof}

\begin{claim}
\label{iclaim}
$P$ forces that every infinite collection of branches of $U$ is eventually dissonant.
\end{claim}

\begin{proof}
It is of course only necessary to treat countably infinite collections of branches of $U$. By the previous claim, such collections are equal to $\{b_\gd\colon \gd\in c\}$ for some countable set $c\subset\gw_2$, and by the $\gs$-closure of the poset $P$, the set $c$ must be in the ground model. Let $p\in P$ be a condition. Strengthening it if necessary, I may assume that $c\subset e_p$ holds. Now, let $q\leq p$ be a condition obtained from $p$ by simply adding $c$ to $b_p$. It is immediate that $q\leq p$ is a condition and $q\Vdash c$ is dissonant at every ordinal larger than $\ga_p$.
\end{proof}

Move to the $P$-generic extension $V[G]$. Consider the Namba forcing $\namba$, its name $\dot x$ for a generic $\gw$-sequence of ordinals cofinal in $\gw_2^V$, and its name $\dot d$ for the infinite collection $\{b_{\dot x(n)}\colon n\in\gw\}$ of branches of $U$. The following claim will conclude the proof of the theorem.

\begin{claim}
\label{iiclaim}
In $V[G]$, the Namba forcing forces $\dot d$ not to be eventually dissonant.
\end{claim}

\begin{proof}
Suppose towards a contradiction that this fails and $T\in\namba$ is a tree and $\ga\in\gw_1$ is an ordinal such that $T\Vdash\dot d$ is dissonant everywhere past $\check\ga$. For simplicity assume that $T$ has empty trunk. For every ordinal $\gb>\ga$ consider a game $H_\gb$ in which Players I and II alternate, creating an increasing sequence $\gg_0<\gd_0<\gg_1<\gd_1<\dots$ of ordinals in $\gw_2$ and Player II wins if his ordinals $\langle \gd_n\colon n\in\gw\rangle$ form a branch through the tree $T$ and for all $n\in\gw$ it is the case that $f(\gb, \gd_n)=0$. The game $H_\gb$ is closed for Player II and therefore determined. If for some ordinal $\gb>\ga$ Player II has a winning strategy, then a standard argument yields a Namba tree $S\subset T$ such that for every ordinal $\gd$ appearing on it, $f(\gb, \gd)=0$. Such a tree $S$ necessarily forces $\dot a$ to be consonant at $\gb$, contradicting the choice of the tree $T$. Thus, it will be enough to derive a contradiction from the assumption that Player I has a winning strategy in the game $H_\gb$ for every ordinal $\gb>\ga$. Note the larger ordinals Player I plays, the better for him. It follows that all of these $\aleph_1$ many strategies for Player I can be compounded into a single one that wins in all the games simultaneously. It also follows from $\aleph_2$-c.c.\ of the poset $P$ that such a universal strategy exists already in the ground model. 

For the rest of the proof, move back to the ground model, let $p\in P$ be a condition, $\dot T$ a $P$-name for the tree $T$ from the previous paragraph, and $\gs$ a strategy for which $p\Vdash\check\gs$ is winning for all the games $H_\gb\colon \gb>\check\ga$ for Player I. I will find a condition $r\leq p$, an ordinal $\gb$ with $\ga<\gb<\ga_r$ and a counterplay $\gg_0<\gd_0<\gg_1<\dots$ against the strategy $\gs$ such that for all $n\in\gw$, $f_r(\gb, \gd_n)=0$ and $r\Vdash\langle\gd_n\colon n\in\gw\rangle$ is a branch through the tree $\dot T$. So, $r$ forces that with this counterplay, Player II won in the game $H_\gb$ against the strategy $\gs$, yielding the final contradiction. 

To construct the condition $r$, let $M$ be a countable elementary submodel of a large structure, let $g\subset P\cap M$ be a generic filter over $M$, and let $q$ be a coordinatewise union of all conditions in $g$. It is immediate that $q$ is a condition stronger than all conditions in $g$, in particular $q\leq p$. Look at the evaluation $\dot T/g$ in the model $M[g]$. Clearly, $M[g]\models \dot T/g$ is a tree in which every node has $\aleph_2$ many immediate successors. Thus, it is possible to find a a counterplay $\gg_0<\gd_0<\gg_1<\dots$ against the strategy $\gs$ such that the sequence $\langle \gd_n\colon n\in\gw\rangle$ is a branch through the tree $\dot T/g$ and cofinal in $M\cap\gw_2$. It follows that $q\Vdash\langle \gd_n\colon n\in\gw\rangle$ is a branch through $\dot T$.

Now, consider a condition $r\leq q$ which is obtained from $q$ by letting $\ga_r=\ga_q+1$, $e_r=e_q$, $f_p\subset f_r$, $a_r=a_q$ and $b_r=b_q$ and such that $f_r(\ga_q, \gd_n)=0$ for all $n\in\gw$. Such a condition $r$ will have the required properties with the ordinal $\gb=\ga_q$. It is necessary to verify that such a condition exists. To see that, observe that each ordinal $\gd_n$ belongs to the set $e_q=M\cap \gw_2$ by the $M$-genericity of the filter $g$. Observe also that $b_q$ is just the set of all countably infinite subsets of $\gw_2$ which belong to the model $M$ by the genericity of the filter $g$. Thus, each set in $b_q$ is bounded in $M\cap\gw_2$ and therefore has finite intersection with the set $\{\gd_n\colon n\in\gw\}$ it immediately follows that the function $f_q$ can be extended to $f_r$ on the domain $(\ga_q+1)\times e_q$ so that the last demand on $f_r$ is satisfied.
\end{proof}

To conclude the proof, move to the $P$-extension. Note that the tree $U$ has size $\aleph_1$ and so $U\in H_{\aleph_2}$. The statement ``any infinite collection of branches of $U$ is eventually dissonant'' is $\Pi_1$ in the parameter $U$ in the structure $\langle H_{\aleph_2}, \in, \gw_1\rangle$, it holds in the $P$-extension by Claim~\ref{iclaim} while it fails in the further Namba extension by Claim~\ref{iiclaim}. This shows that in the $P$-extension, the bounded forcing axiom for Namba forcing fails.

\bibliographystyle{plain} 
\bibliography{odkazy}

\end{document}